\documentclass[11pt,reqno]{amsart} 
\usepackage{amsmath,tikz}
\usepackage[utf8]{inputenc}
\usepackage{hyperref}
\usepackage{amsmath}
\usepackage{amssymb}
\usepackage{amsthm}
\usepackage{mathtools}
\usepackage{verbatim}
\usepackage{xcolor}
\usepackage{fullpage}
\usepackage{soul}
\usepackage{caption}
\usepackage{graphicx}
\graphicspath{ {./images/} }
\usepackage{multicol}
\usepackage{setspace}

\DeclareMathOperator\inv{\text{inv}}
\DeclareMathOperator\negative{\text{neg}}
\DeclareMathOperator\nsp{\text{nsp}}

\newcommand{\id}{\textbf{e}}

\newtheorem{theorem}{Theorem}[section]
\newtheorem{corollary}[theorem]{Corollary}
\newtheorem{proposition}[theorem]{Proposition}

\newtheorem{lemma}[theorem]{Lemma}
\newtheorem{remark}[theorem]{Remark}

\theoremstyle{definition}
\newtheorem{definition}[theorem]{Definition}

\newtheorem{example}[theorem]{Example}

\newcommand*\xbar[1]{%
  \hbox{%
    \vbox{%
      \hrule height 0.5pt % The actual bar
      \kern0.5ex%         % Distance between bar and symbol
      \hbox{%
        \kern-0.1em%      % Shortening on the left side
        \ensuremath{#1}%
        \kern-0.1em%      % Shortening on the right side
      }%
    }%
  }%
}

\title[Metrics on Signed Permutations with the Same Peak Set]{Metrics on Signed Permutations \\ with the Same Peak Set}

\author{Kayla Andrus}
\address[K.~Andrus]{San Diego State University, San Diego, CA 92115}
\email{\textcolor{blue}{\href{mailto:}kayla@andrus-family.com}}

\author{Nathaniel Larsen}
\address[N.~Larsen]{San Diego State University, San Diego, CA 92115}
\email{\textcolor{blue}{\href{mailto:}nathanplarsen@gmail.com}}

\author{Alyssa MacLennan}
\address[A.~MacLennan]{San Diego State University, San Diego, CA 92115}
\email{\textcolor{blue}{\href{mailto:}alyssamaclennan@gmail.com}}

\author{Gordon Rojas Kirby}
\address[G.~Rojas Kirby]{Department of Mathematics and Statistics, San Diego State University, CA 92182}
\email{\textcolor{blue}{\href{mailto:gkirby@sdsu.edu}{gkirby@sdsu.edu}}}
\thanks{}
 
\author{Mariana Smit Vega Garcia}
\address[M.~Smit Vega Garcia]{Department of Mathematics, Western Washington University, Bellingham, WA 98225}
\email{\textcolor{blue}{\href{mailto:}{smitvem@wwu.edu}}}
\thanks{M.S.V.G was partially supported by the NSF grant DMS-2348739 and a Karen EDGE fellowship}

\author{Christian Vicars}
\address[C.~Vicars]{San Diego State University, San Diego, CA 92115}
\email{\textcolor{blue}{\href{mailto:}cvicars21@gmail.com}}

\begin{document}
\begin{abstract}
Let $S^B_n$ be the Coxeter group of type B. We denote the set of indices where $\sigma\in S^B_n$ has a peak as $Peak(\sigma)$ and let $P^{B}(S;n)=\{\sigma \in S^{B}_n~|~ Peak(\sigma)=S\}$. In \cite{metrics}, Diaz-Lopez, Haymaker, Keough, Park and White considered metrics for unsigned permutations with the same peak set. In this paper, we generalize their result by studying Hamming, $l_{\infty}$, and the word metrics on $P^{B}(S;n)$ for all $S$. We also determine
the minimum and maximum possible values that these metrics can achieve in these subsets of $S^B_n$.
\end{abstract}
\maketitle

\section{Introduction}
In this paper, we consider signed permutations of the Coxeter group of type $B$. Our main results determine the minimum and maximum distances that can be achieved between two signed permutations with the same peak set. We use $S^B_n$ to denote the \textbf{Coxeter group of type $B_n$}, that is, $S^B_n$ consists of all bijections from the set $\{ -n, \ldots, -2,-1 , 1, 2, \ldots , n \}$ to itself, such that if $\sigma(i)=j$, then $\sigma(-i) = -j$. To simplify the notation, we often write $-j=\bar{j}$, and write our permutations in one-line notation as $\sigma = \sigma(1)\sigma(2)\cdots\sigma(n)$. Notice that if we take the absolute value of every entry in the one-line notation of a signed permutation of $S^B_n$, all values of $\{1,2,...,n\}$ will appear exactly once. We say that a permutation $\sigma = \sigma(1)\sigma(2)\cdots\sigma(n) \in S^B_n$ has a \textbf{peak} at position $i \in \{2,3 \ldots , n-1 \} $ when $\sigma(i-1) < \sigma(i) > \sigma(i+1)$. We then define the \textbf{peak set} of $\sigma$, denoted $Peak(\sigma)$, as the set of all the indices in which $\sigma$ has a peak. Finally, we notate the collection of signed permutations that possess the same peak set $S$ as $P^{B}(S;n)=\{\sigma \in S^{B}_n~|~ Peak(\sigma)=S\}$. When there exists at least one $\sigma\in S^B_n$ with $Peak(\sigma)=S$, we say $S$ is an admissible peak set.

To exemplify these definitions, consider $\sigma = \bar{5}2\bar{4}3\bar{1} \in S^B_5$. This means that $\sigma$ maps $1 \mapsto \bar{5}, 2 \mapsto 2, 3 \mapsto \bar{4}, 4 \mapsto 3,$ and $5 \mapsto \bar{1}$. 
Since $\bar{5}<2>\bar{4}$, $i=2$ is a peak. Similarly, $i=4$ is a peak, since $\bar{4} < 3 > \bar{1}$. Since $2\not< \bar{4} \not> 3$, $i=3$ is not a peak. Thus, $\sigma$ has peak set $Peak(\sigma)=\{2,4\}$. Given $S=\{2,4\}$, the set $P^{B}(S;5)=\{\sigma \in S^{B}_5\mid Peak(\sigma)=S \}$ consists of all permutations with peak set $\{2,4\}$.

Considering the Coxeter group of type $A_{n-1}$, the symmetric group $S_n$ consisting of unsigned permutations, the analogously defined sets $P^A(S;n)$ were first studied by Nyman in \cite{N}. Afterwards, Billey, Burzy and Sagan \cite{BBS} studied the cardinality of the set $P^A(S;n)$.  Their results were then extended to permutations
of type B by Castro-Velez et al. \cite{CVal}. Numerous applications led to the study of distances in $S_n$: in statistics, for example \cite{C}, in coding theory, \cite{BCD}, in computing \cite{K}. The interested reader is also pointed to the references therein, along with \cite{KG}, and the survey \cite{DH}.

We consider three different metrics on the Coxeter group of type $B$: the Hamming metric, the $l_{\infty}$-metric, and the word metric. We determine the maximum and minimum values that each metric can obtain in each subset $P^{B}(S;n)$ for an admissible peak set $S$. Note that the minimum attained by any metric is 0, accordingly the minimum value we consider is between distinct permutations.
Our main results are the following.

\begin{enumerate}
    \item[(1)] Theorem \ref{thm:MinMetrics-Signed-Same-Peak-Set} proves that if $S$ is an admissible peak set in $S^B_n$ with $n\ge 2$, then the minimum for each of the three metrics in Definitions \ref{inftyB}, \ref{HammingB} and \ref{WordB} is 
\begin{itemize}
        \item $\min\{d_H(\sigma,\rho) \ : \ \sigma,\rho\in P^B(S;n), \ \sigma\neq \rho\} = 1$.
        \item $\min\{d_\ell(\sigma,\rho) \ : \ \sigma,\rho\in P^B(S;n), \ \sigma\neq \rho\} = 1$.
        \item $\min\{d_W(\sigma,\rho) \ : \ \sigma,\rho\in P^B(S;n), \ \sigma\neq \rho\}= 1$.
    \end{itemize}
    \item[(2)] Theorem \ref{prop:MaxMetrics-Signed-Same-Peak-Set} proves that if $S$ is an admissible peak set in $S^B_n$ with $n\ge 2$, then the maximum for each of the three metrics in Definitions \ref{inftyB}, \ref{HammingB}, \ref{WordB} is
\begin{itemize}
        \item max($d_H(P^B(S;n)$)) = n.
        \item 
        $
        \text{max}(d_\ell(P^B(S;n)) =
            \bigg\{ \begin{array}{lr}
                 2n-1 ,&\text{if } \{2, n-1\} \subset S  \\
                 2n   ,&\text{otherwise}.
            \end{array}
        $
        \item max($d_W(P^B(S;n))$) = $n^2 - |S|.$
    \end{itemize}
\end{enumerate}

Our findings extend \cite{metrics}, where the authors determined the minimum and maximum possible values that the Hamming, the $l_{\infty}$ and the Kendall-Tau metrics can achieve on unsigned permutations with the same peak set. Notice that while we consider the word metric, we note in Remark \ref{Kendall} that in the context of unsigned permutations the word metric coincides with the Kendall-Tau metric considered in \cite{metrics}.

\section{Metrics on permutations}   

We begin this section by defining unsigned permutations, peaks, and peak sets. We also define and give examples of left and right multiplication for permutations. 

A \textbf{permutation} is a bijective function from $[n]:=\{1, 2, \ldots, n \}$ to $[n]$. In one-line notation, $\sigma = 2341$ is the permutation for which  $\sigma(1) =2,$ $\sigma(2) =3$, $\sigma(3)= 4$, and $\sigma(4) =1$.
A permutation $\sigma$ has a \textbf{peak} at index $i\in\{2,\ldots, n-1\}$ if $\sigma(i-1) < \sigma(i) > \sigma(i+1)$. Accordingly, the \textbf{peak set} of a permutation $\sigma$, denoted $Peak(\sigma)$, is the set of indices at which $\sigma$ has a peak. For example, $\sigma = 2435176$ has peak set $Peak(\sigma) = \{2,4,6 \}$. We notate the collection of permutations that possess the same peak set $S$ as $P(S;n)=\{\sigma \in S^B_n~|~ Peak(\sigma)=S\}$.

We now define two of our metrics of interest: the \textit{${l_\infty}$-\textbf{metric}}, and the \textbf{Hamming metric}. We will use unsigned permutations $\sigma = 21543$ and $\rho = 32415$ as illustrative examples. Notice that $\sigma$ and $\rho$ have the same peak set. That is, $Peak(\sigma) = Peak(\rho) =\{3\}$. 

\begin{definition}\label{linfty}
    Let $d_\ell$, denoting the \textit{${l_\infty}$-\textbf{metric}}, be the map $d_\ell: S_n\times S_n \rightarrow [0, \infty)$ such that $d_\ell(\sigma,\rho) = $max$\{|\sigma(i) - \rho(i)| \ : \ i\in [n]\} $.
\end{definition}

Considering $\sigma = 21543$ and $\rho = 32415$, $d_\ell (\sigma, \rho) = |\sigma(4) - \rho(4)| = 3$. 

\begin{definition}\label{Hamming}
    Let $d_{H}$, denoting the \textbf{Hamming metric}, be the map $d_{H}:S_n\times S_n\to [0,\infty)$ such that $d_H(\sigma,\rho)$ is the number of indices where $\sigma$ and $\rho$ differ. That is, if $\sigma = \sigma(1)\sigma(2)...\sigma(n)$ and $\rho = \rho(1)\rho(2)...\rho(n)$ then 
    $$
    d_{H}(\sigma,\rho)=|\{i\in [n] \ : \ \sigma(i)\neq\rho(i)\}|.
    $$  
\end{definition}

Since all entries of $\sigma = 21543$ and $\rho = 32415$ differ, $d_H(\sigma,\rho) = 5$.  \\

To define our third metric of interest, the \textbf{word metric}, we will first introduce a few auxiliary definitions, including that of transpositions. Transpositions provide an easy method for generating permutations.

\begin{definition}[Transpositions]\label{unsignedtransp}
    A \textbf{transposition} of $S_n$ is a permutation that swaps two elements $i$ and $j$ and fixes all other elements of $[n]$. We say that a transposition of $S_n$ is an \textbf{adjacent transposition} when the two transposed elements are $i$ and $i+1$ for some $i\in [n-1]$ and denote it with $s_i$.
\end{definition}

Recall that the symmetric group acts on itself via left-multiplication and right-multiplication. We remind the reader that multiplying a permutation $\sigma$ on the left by an adjacent permutation swapping $i$ and $i+1$ results in a \textbf{value swap} of $i $ and $i+1$ in the permutation. Multiplying a permutation $\sigma$ on the right by an adjacent permutation swapping $i$ and $i+1$ results in an \textbf{index swap} of the values at indices $i $ and $i+1$ in the permutation. As is well-known, every permutation of $[n]$ can be obtained by a sequence of adjacent index swaps or adjacent value swaps. In fact, we define length and the word metric in terms of adjacent permutations.

\begin{definition}
    Given $\pi\in S_n$, write $\pi=\tau_1\tau_2\cdots \tau_k$. If $k$ is minimal among all such expressions for $\pi$, then we say that $\pi$ has \textbf{length} $k$ and write $\ell(\pi)=k$.
\end{definition}

\begin{definition}\label{WordMetricDef}
    Let $d_W:S_n\times S_n\to [0,\infty)$ be defined by $d_W(\sigma,\pi)=\ell(\pi^{-1}\sigma)$, where $\ell$ is the length function for $S_n$ with respect to adjacent transpositions. We say that $d_W$ is the \textbf{word metric} on $S_n$ with respect to the generating set of adjacent transpositions. That is, $d_W(\sigma,\pi)$ is the minimum number of adjacent transpositions required to transform $\sigma$ into $\pi$ via right multiplication.
\end{definition}

\begin{remark}\label{Kendall}
It is well known that the length of a permutation can also be computed as the number of inversions of said permutation, see \cite{BB} Proposition 1.5.2 for example, where the number of inversions of a permutation $\pi$ is defined as $\inv(\pi)=|\{(i,j)\ :\ 1\le i<j\le n, \ \pi(i)>\pi(j)\}|$. That is, $\ell(\pi)=\inv(\pi)$. Hence, the distance between two permutations $\sigma$ and $\rho$ can be computed as the number of deranged pairs, where a deranged pair of $\sigma$ and $\rho$ is a pair $(i,j)$ with $1\le i<j\le n$ satisfying $(\sigma(i)-\rho(i))(\sigma(j)-\rho(j)) <0$. We note that the authors in \cite{metrics} refer to $d_W$ as the Kendall-Tau metric, which they denote as $d_K$.
\end{remark}

For $\sigma = 21543$ and $\rho = 32415$, we can compute $d_W(\sigma , \rho) = 4$, as $(1,5),(2,4),(3,5),(4,5)$ all satisfy $(\sigma(i)-\rho(i))(\sigma(j)-\rho(j)) <0$, and thus are all deranged pairs.

We now summarize the main result from \cite{metrics}, which determines the maximum and minimum value of each of the previous three metrics for unsigned permutations with the same peak set. This is the result that we generalize to signed permutations in Section \ref{sec:MetricsSigned}.

\begin{theorem}[Proposition 3.2, Theorems 3.4, 3.5, 3.6 of \cite{metrics}] Given $S\subset [n]$, the maximum and minimum values for the Hamming, $l_{\infty}$, and word metric on unsigned permutations with peak set $S$ are
    \begin{align*}
\min(d_H(P(S;n))) &= 2, \quad \max(d_H(P(S;n))) = n, \\
\min(d_\ell(P(S;n))) &= 1, \quad  \max(d_\ell(P(S;n))) =
            \bigg\{ \begin{array}{lr}
                 n-2 ,&\text{if } \  \{2, n-1\} \subset S  \\
                 n-1   ,&\text{otherwise} 
            \end{array} \\
\min(d_K(P(S;n))) &= 1, \quad \max(d_K(P(S;n))) = \binom{n}{2} -2|S|.
\end{align*}
\end{theorem}

%%%%%%%%%%%%%%%%%%%%%%%%%%%%%%%%%%%%%%%%%%%%%%%%%%%%
%Section
%%%%%%%%%%%%%%%%%%%%%%%%%%%%%%%%%%%%%%%%%%%%%%%%%%%%
\section{Signed permutations}

In this section, we discuss some properties of the group of signed permutations $S_{n}^{B}$ and generalize the $l_{\infty}$, the Hamming and word metrics to $S^B_n$. 
The tools presented in this Section will be useful in Section \ref{sec:MetricsSigned}, where we compare the extremal metric values on unsigned permutations from \cite{metrics} with our results for signed permutations.

First, recall that $S^B_n$ is a group with multiplication defined as composition.

\begin{example}
    Given $\tau = 12\bar{4}\bar{3}$ and $\sigma = 3\bar{2}1\bar{4}$, we compute the product $\tau\sigma$. We view the numbers in $\sigma$ as indicating positions in $\tau$. For example, the first value of $\sigma$ is $3$, and at index $3$ of $\tau$ there is a $-4$. That is, $\tau\sigma(1)=-4$ so that the first entry of the one-line notation for $\tau\sigma$ is $\bar 4$. The second value of $\sigma$ is $-2$, which we can think of as indicating index $2$ of $\tau$ with a change of sign, i.e. $\tau\sigma(2)=-2$. Continuing in this fashion yields $\tau\sigma=\bar4\bar213$. We note that the product $\sigma\tau=3\bar24\bar1$ so that $S_n^B$ is not commutative.
\end{example}

Another property of the group $S^B_n$  is that every element has an \textbf{inverse}. For example, to compute the inverse of $\sigma = \bar{3}1\bar{2}4$, where $\sigma(1) =-3,~\sigma(2) =1,~\sigma(3)= -2 ,~\sigma(4) =4$, we negate all indices and their values and write this backwards: since $\sigma(-1) =3,~\sigma(-2) = -1,~ \sigma(-3)= 2 ,~\sigma(-4) = -4$, the inverse becomes $\sigma^{-1} = \bar{4}2\bar{1}3$. 

The group of unsigned permutations $S_{n}$, as well as the group of signed permutations $S_{n}^{B}$, are both Coxeter groups and come with a prescribed group of generators. In fact, the group of unsigned permutations is a subgroup of the group of signed permutations, so $S_{n}$ is embedded in $S_{n}^{B}$ for a given $n$. We will later use this fact to help us determine extremal values of our metric in Section \ref{sec:MetricsSigned}. 

The generators we will use to create signed permutations from the identity are transpositions, together with an additional generator that swaps the sign of the first entry.

To be able to define the word metric on $S^B_n$, we need to fix a set of generators of $S^B_n$. As mentioned above, we will use adjacent transpositions and negation of the first entry:

\begin{definition}\label{SignedGeneratorDef}
    In $S_n^B$, we define the \textbf{Coxeter generators}, denoted $s_i^B$ and $s_0^B$, as follows. For $1\leq i<n$, let $s_i^B$ be the permutation that swaps $i$ with $i+1$ and $-i$ with $-(i+1)$. These act just like the adjacent transpositions we are familiar with from the unsigned permutation group $S_n$, except they are now extended symmetrically to include negative values. In one-line notation, $s_i^B$ looks like the identity permutation with entries $i$ and $i+1$ transposed. For instance, $s_2^B$ is denoted $1 3 2 4 \ldots  n$. We also define the generator $s_0^B$ that swaps $1$ and $-1$. Its one-line notation is $\bar{1}23 \ldots n$. 
\end{definition}

It is straightforward to verify that $S_n^B$ can be generated by $s_0^B$, $s_1^B$, $\dots$, $s_{n-1}^B$. The \textbf{length} of a signed permutation $\sigma$, denoted $\ell_B(\sigma)$, is defined as the minimum number of generators needed in a product equal to $\sigma$, generalizing Definition \ref{WordMetricDef}. For example, in the setting of signed permutations, an expression for $\sigma = 21\bar{4}3$ using a minimal number of generators is $s_3^B s_2^B s_1^B s_0^B s_1^B s_2^B s_1^B$. Thus, $\ell_B(\sigma)=7$. It is also worth mentioning that via Proposition 8.1.1 of \cite{BB} , $\ell_B(\sigma)=7=inv_B(\sigma)$.

We restate Definitions \ref{linfty}, \ref{Hamming} and \ref{WordMetricDef} for $S_n^B$

\begin{definition}\label{inftyB}
    Let $d_\ell$, denoting the \textit{${l_\infty}$-\textbf{metric}}, be the map $d_\ell: S_n^B\times S_n^B \rightarrow [0, \infty)$ such that $d_\ell(\sigma,\rho) = $max$\{|\sigma(i) - \rho(i)| \ : \ i\in [n]\} $.
\end{definition}

\begin{definition}\label{HammingB}
    Let $d_{H}$, denoting the \textbf{Hamming metric}, be the map $d_{H}:S_n^B\times S_n^B\to [0,\infty)$ such that if $\sigma = \sigma(1)\sigma(2)...\sigma(n)$ and $\rho = \rho(1)\rho(2)...\rho(n)$ then 
    $$
    d_{H}(\sigma,\rho)=|\{i\in [n] \ : \ \sigma(i)\neq\rho(i)\}|.
    $$  
\end{definition}

\begin{definition}\label{WordB}
    Let $d_W:S_n^B\times S_n^B\to [0,\infty)$ be the \textbf{word metric} in $S_n^B$, which is defined by $d_W(\sigma,\pi)=\ell_B(\pi^{-1}\sigma)$.
\end{definition}

\begin{remark}\label{wordremark}
    The word metric is left-invariant, that is, given $\sigma, \rho, \tau \in S^B_n$, $d_W(\sigma,\rho) = d_W(\tau\sigma,\tau\rho)$. Moreover, calling $\textbf{e}=12\ldots n$, $d_W(\textbf{e},\gamma)=\ell_B(\gamma)$ for any $\gamma\in S^B_n$. Proposition 8.1.1 from \cite{BB}, for example, shows that $\ell_B(\gamma)=\inv_B(\gamma)$, where 
\[
\inv_B(\gamma)=\negative(\gamma(1),\ldots,\gamma(n))+\inv(\gamma(1),\ldots,\gamma(n))+\nsp(\gamma(1),\ldots,\gamma(n)).\] 
Here 
$\negative(\gamma(1),\ldots,\gamma(n))$ is the number of negative entries of $\gamma$ when restricted to $[n]$, 
$\inv(\gamma(1),\ldots,\gamma(n))$ is the number of inversions of $\gamma$ (that is, pairs $(i,j)$ with $1\le i<j\le n$ so that $\gamma(i)>\gamma(j)$), and $\nsp(\gamma(1),\ldots,\gamma(n))$ is the number of negative sum pairs of $\gamma$ (that is, $(i,j)$ with $1\le i<j\le n$ for which $\gamma(i)+\gamma(j)<0$).

\end{remark}

%%%%%%%%%%%%%%%%%%%%%%%%%%%%%%%%%%%%%%%%%%%%%%%%%%%%
%Section
%%%%%%%%%%%%%%%%%%%%%%%%%%%%%%%%%%%%%%%%%%%%%%%%%%%%
\section{Main results}\label{sec:MetricsSigned}

In this section, we prove Proposition \ref{prop:MaxMinMetrics-Signed} and Theorems \ref{thm:MinMetrics-Signed-Same-Peak-Set} and \ref{prop:MaxMetrics-Signed-Same-Peak-Set}, where we determine the minimum and maximum value attained by each of our three metrics in $S^B_n$ and in $P^B(S;n)$ for an admissible peak set $S$. We initially used SageMath to generate enough data to motivate conjectures pertaining to the the maximum and minimum attained by each metric in $S_n^B$ and $P^B(S;n)$ for arbitrary $S$. 
This data is in the tables in Figure \ref{fig:data_tables}.  We were then able to prove our conjectures. Proposition \ref{prop:MaxMinMetrics-Signed} describes the extremal values attained by our three metrics in $S_n^B$. Theorem \ref{thm:MinMetrics-Signed-Same-Peak-Set} and Theorem \ref{prop:MaxMetrics-Signed-Same-Peak-Set}, our main results, describe the minimum and maximum values of those metrics on $P^B(S;n)$ for arbitrary $S$.

\begin{figure}[h]
\begin{center}
\begin{tabular}{|p{3cm}|p{1cm}|p{1cm}|p{1cm}|p{1cm}|p{1cm}|p{1cm} |}
    \hline
    \textbf{Peak Set, n=3} & \textbf{Word min} & \textbf{Word max} & \textbf{Ham min} & \textbf{Ham max} & \textbf{l-inf min} & \textbf{l-inf max}  \\
    \hline
    $\emptyset$ & 1 & 9 & 1 & 3 & 1 & 5  \\
    \hline
    \{2\} & 1 & 8 & 1 & 3 & 1 & 5   \\
    \hline
    \textbf{Overall} & 1 & 9 & 1 & 3 & 1 & 5  \\
    \hline
    
    \hline
    \textbf{Peak Set, n=4} & \textbf{Word min} & \textbf{Word max} & \textbf{Ham min} & \textbf{Ham max} & \textbf{l-inf min} & \textbf{l-inf max}  \\
    \hline
    $\emptyset$ & 1 & 16 & 1 & 4 & 1 & 8  \\
    \hline
    \{2\} & 1 & 15 & 1 & 4 & 1 & 8  \\
    \hline
    \{3\} & 1 & 15 & 1 & 4 & 1 & 8 \\
    \hline
    \textbf{Overall} & 1 & 16 & 1 & 4 & 1 & 8 \\
    \hline
    
    \hline
    \textbf{Peak Set, n=5} & \textbf{Word min} & \textbf{Word max} & \textbf{Ham min} & \textbf{Ham max} & \textbf{l-inf min} & \textbf{l-inf max} \\
    \hline
     $\emptyset$ & 1 & 25 & 1 & 5 & 1 & 10  \\
     \hline
     \{2\} & 1 & 24 & 1 & 5 & 1 & 10 \\
     \hline
     \{3\} & 1 & 24 & 1 & 5 & 1 & 10 \\
     \hline
     \{4\} & 1 & 24 & 1 & 5 & 1 & 10  \\
     \hline
     \{2,4\} & 1 & 23 & 1 & 5 & 1 & 9  \\
     \hline
     \textbf{Overall} & 1 & 25 & 1 & 5 & 1 & 10  \\
     \hline
    
    \hline
    \textbf{Peak Set, n=6} & \textbf{Word min} & \textbf{Word max} & \textbf{Ham min} & \textbf{Ham max} & \textbf{l-inf min} & \textbf{l-inf max} \\
    \hline
    $\emptyset$ & 1 & 36 & 1 & 6 & 1 & 12 \\
    \hline
    \{2\} & 1 & 35 & 1 & 6 & 1 & 12 \\
    \hline
    \{3\} & 1 & 35 & 1 & 6 & 1 & 12 \\
    \hline
    \{4\} & 1 & 35 & 1 & 6 & 1 & 12 \\
    \hline
    \{5\} & 1 & 35 & 1 & 6 & 1 & 12 \\
    \hline
    \{2,4\} & 1 & 34 & 1 & 6 & 1 & 12 \\
    \hline
    \{2,5\} & 1 & 34 & 1 & 6 & 1 & 11 \\
    \hline
    \{3,5\} & 1 & 34 & 1 & 6 & 1 & 12 \\
    \hline
    \textbf{Overall} & 1 & 36 & 1 & 6 & 1 & 12 \\
    \hline
    \end{tabular}
\end{center}
\caption{Data tables for metrics of signed permutations with same peak set.}
\label{fig:data_tables}
\end{figure}

\begin{proposition}\label{prop:MaxMinMetrics-Signed}
    Given $n \geq 2$, we have
\begin{itemize}
        \item $\min\{d_H(\sigma,\rho) \ : \ \sigma,\rho\in S^B_n, \ \sigma\neq \rho\} = 1$, $\max(d_H(S^B_n)) = n$.
        \item $\min\{d_\ell(\sigma,\rho) \ : \ \sigma,\rho\in S^B_n, \ \sigma\neq \rho\}= 1$, $\max(d_\ell(S^B_n)) = 2n$.
        \item $\min\{d_W(\sigma,\rho) \ : \ \sigma,\rho\in S^B_n, \ \sigma\neq \rho\} = 1$, $\max(d_W(S^B_n)) = n^2$.
    \end{itemize}
    \end{proposition}
    
\begin{proof} By definition, given $\tau,\pi\in S^B_n$, $d_H(\tau,\pi)\le n$. To show that $n$ is attained, let $\textbf{e} = 123\ldots n$ and $\rho = 234 \ldots n1$. Then $d_H(\textbf{e},\rho)=n$, as desired. To prove $\min(d_H(S^B_n)) = 1$, notice first that if $\tau,\pi\in S^B_n$ and $\tau\neq\pi$, then $d_H(\tau,\pi)\ge 1.$ Now let $\rho = \bar{1}23 \hdots n$. Then $d_H(\textbf{e},\rho)=1$, as desired.

    For the $l_\infty$-metric, by definition one sees that if $\tau,\pi\in S^B_n$ and $\tau\neq\pi$, then $1\le d_\ell(\tau,\pi)\le 2n.$ Letting  $\rho = 213 \ldots n$, we see that $d_\ell(\textbf{e},\rho)=1$. Defining $\tau=n 
    \; (n-1) \ldots 1$ and $\pi = \bar{n} 
    \; (n-1) \ldots 1$, for example, we obtain $d_\ell(\tau,\pi)=2n$.
    
Finally, considering the word metric, we notice that if $\tau,\pi\in S^B_n$ and $\tau\neq\pi$, then $1\le d_W(\tau,\pi)$. Letting $\rho = \bar{1}23 \ldots n$, we obtain $d_W(\textbf{e},\rho)=1$. Now, since $d_W(\sigma\tau)=\ell_B(\pi^{-1}\sigma)=\text{inv}_B(\pi^{-1}\sigma)$ (see Remark \ref{wordremark}), we obtain $d_W(S^B_n)\le n+2|(i,j) \ : \ 1\le i<j\le n\}|=n+2\frac{n(n-1)}{2}=n^2.$ To show that the maximum is attained, let  $\bar{\textbf{e}}=\bar{1}\bar{2}\ldots\bar{n}$. We see that $d_W(\textbf{e},\bar{\textbf{e}})=\ell_B(\textbf{e}^{-1}\bar{\textbf{e}})=\ell_B(\bar{\textbf{e}})=\sum_{k=1}^n(2k-1)=n^2$, as desired.

\end{proof}
\begin{corollary} For signed permutations with empty peak set, $\max(d_W(P^B(\emptyset;n)))=n^2$.
\end{corollary}

\begin{proof}
    This follows immediately from Theorem \ref{prop:MaxMinMetrics-Signed}, as $d_W(\textbf{e},\bar{\textbf{e}})=n^2$ and both $\textbf{e},\bar{\textbf{e}}\in P^B(\emptyset;n)$.
\end{proof}

The following lemma and definition will be used to prove our remaining main results. 

\begin{lemma}
    (As Lemma 3.1 in \cite{metrics}) Let $S$ be an admissible peak set and $\sigma \in P^B(S;n)$. For any $i \in \{2,3, \hdots ,n-2,n-1\}$, if $i$ and $i+1$ do not appear consecutively in $\sigma$, swapping $i$ and $i+1$ creates a permutation $\sigma'$ with the same peak set as $\sigma$, i.e. $\sigma' \in P^B(S;n)$. Also, swapping $\overline{n}$ and $\overline{n-1}$ produces a permutation with the same peak set as $\sigma$.  
\end{lemma}

\begin{definition}\label{es}
    As before, let $\id$ be the identity permutation $1 2 \ldots n$ and $\bar \id=\bar1\bar2\cdots\bar n$ the identity permutation with all values barred. We use a superscript $*$ to indicate writing a permutation in reverse. That is, $\id^*=n(n-1)\cdots1$ and $\bar \id^*=\bar n(\overline{n-1})\cdots\bar1$

    For an admissible peak set $S$, define $\textbf{e}[S]$ and $\bar \id^*[S]$ as the permutations obtained by swapping the entries $k$  and $k+1$ in \textbf{e} and $\bar \id^*$ for each $k \in S$. Similarly, let $\id^*[S]$ and $\bar \id[S]$ be the permutations obtained by swapping the entries $k-1$ and $k$ in \textbf{e}* and $\bar \id$, for each $k \in S$. Since any admissible set $S$ has no consecutive entries, these permutations are well-defined because the order of the swaps does not matter. More explicitly, we have that for $i \in \{1,2, \ldots , n\}$,

\arraycolsep=.5cm
\[\begin{array}{ll}
        
    \,\ \textbf{e} [S](i) = 
    \begin{cases}
        i+1 & \text{if } i \in S\\
        i-1 & \text{if } i-1\in S\\
        i   & \text{otherwise}
    \end{cases}
&
    
    \bar \id^*[S](i) = 
    \begin{cases}
        (-n-1+i)+1 & \text{if } i \in S\\
        (-n-1+i)-1 & \text{if } i-1\in S\\
        -n-1+i   & \text{otherwise}
    \end{cases}
    \vspace{12pt}
    \\
    
    \id^* [S](i) = 
    \begin{cases}
        (n+1-i)+1 & \text{if } i \in S\\
        (n+1-i)-1 & \text{if } i-1\in S\\
        n+1-i   & \text{otherwise}
    \end{cases}

&   \,\ \bar \id [S](i) = 
    \begin{cases}
        -i+1 & \text{if } i \in S\\
        -i-1 & \text{if } i-1\in S\\
        -i   & \text{otherwise}
    \end{cases}
\end{array}\]
    \end{definition}
    
\begin{example}
    Let $n=6$ and $S=\{2,5\}$. We have
    $\textbf{e}[S]=132465$.

    If $n=9$ and $S = \{ 2, 5, 7\}$, we have $\textbf{e}[S] = 132465879, \bar \id^*[S] = \bar{9} \bar{7} \bar{8} \bar{6} \bar{4}  \bar{5} \bar{3} \bar{2} \bar{1}$, $\id^*[S] = 897563421$, and $\bar \id[S] = \bar{2} \bar{1} \bar{3}\bar{5} \bar{4}\bar{7} \bar{6} \bar{8} \bar{9}$.    
    \end{example}

We will now prove our result concerning the minimum of the three metrics in $P^B(S;n)$.

\begin{theorem}\label{thm:MinMetrics-Signed-Same-Peak-Set}
    Let $S$ be an admissible peak set in $S^B_n$ with $n\ge 2$. Then 
\begin{itemize}
        \item $\min\{d_H(\sigma,\rho) \ : \ \sigma,\rho\in P^B(S;n), \ \sigma\neq \rho\} =1$.
        \item $\min\{d_\ell(\sigma,\rho) \ : \ \sigma,\rho\in P^B(S;n), \ \sigma\neq \rho\} = 1$.
        \item $\min\{d_W(\sigma,\rho) \ : \ \sigma,\rho\in P^B(S;n), \ \sigma\neq \rho\} = 1$.
    \end{itemize}
\end{theorem}

\begin{proof}
   Let $S$ be an admissible set. We start with the Hamming and word metrics. By definition, $\textbf{e}[S] \in P^B(S;n)$. Define $\tau$ to be $\textbf{e}[S]$, with its first entry negated. Then $\tau\in P^B(S;n)$ and $d_H(\textbf{e}[S],\tau)=1=d_W(\textbf{e}[S],\tau)$.

Focusing now on the $l_{\infty}$ distance, notice that $\textbf{e}[S](1)=1$. We have two possibilities: either $\textbf{e}[S](2)=2$, or $\textbf{e}[S](2)\neq 2$. If $\textbf{e}[S](2)=2$, we conclude $2\notin S$, and we can swap the values $1$ and $2$ in $\textbf{e}[S]$ to create a new permutation $\rho\in P^B(S;n)$. If $\textbf{e}[S](2)\neq 2$, Lemma 3.1 from \cite{metrics} implies that we can swap the values $1$ and $2$ in $\textbf{e}[S]$ to create a new permutation $\rho\in P^B(S;n)$. In either case, we have $\textbf{e}[S], \rho\in P^B(S;n)$ and $d_{\ell}(\textbf{e}[S],\rho)=1$.
\end{proof}

\begin{example} To illustrate the proof technique of Theorem 4.6, consider the following example. Let $S = \{2,4\}$ with $n=5$ and suppose $\tau$ and $\rho$ are defined as in the proof of Theorem \ref{thm:MinMetrics-Signed-Same-Peak-Set}. Notice that $\textbf{e}[S] = 13254\in P^B(S;5)$, $\tau=\bar{1}3254\in P^B(S;5)$ and $d_H(\textbf{e}[S],\tau)=1=d_W(\textbf{e}[S],\tau)$. Also, $\rho=23154\in P^B(S;5)$ and $d_{\ell}(\textbf{e}[S],\rho)=1$.  
\end{example}

Having proven our main result with respect to the minimum of our metrics in $P^B(S;n)$, we now focus on the maximum. 

\begin{theorem}\label{prop:MaxMetrics-Signed-Same-Peak-Set}
      Let $S$ be an admissible peak set in $S^B_n$ with $n\ge 2$. Then 
\begin{enumerate}
        \item[(a)] $\max(d_H(P^B(S;n))) = n$.
        \item[(b)] 
        $
        \max(d_\ell(P^B(S;n)) =
            \bigg\{ \begin{array}{lr}
                 2n-1 ,&\text{if } \ \{2, n-1\} \subset S  \\
                 2n   ,&\text{otherwise}.
            \end{array}
        $
        \item[(c)] $\max(d_W(P^B(S;n))) = n^2 - |S|$.
    \end{enumerate}
\end{theorem}

We will prove (a), (b) and (c) separately.

\begin{proof}[Proof of Theorem \ref{prop:MaxMetrics-Signed-Same-Peak-Set}(a)] By Proposition \ref{prop:MaxMinMetrics-Signed}, we know $\max(d_H(P^B(S;n))) \le n$. By Theorem 3.6 in \cite{metrics}, we know $\max(d_H(P(S;n))=n$, from what it follows immediately that $\max(d_H(P^B(S;n))=n$.
 \end{proof}

\begin{example}
     To illustrate the proof technique of Theorem \ref{prop:MaxMetrics-Signed-Same-Peak-Set}(a), we use the construction in the proof of Theorem 3.6 in \cite{metrics} in the following example. Let $S=\{2,4\}$, $n=6$, $\sigma = 132546$, $\rho = 461325$. We have $\sigma, \rho\in P^B(S;6)$ and $d_H(\sigma,\rho)=6$.
\end{example}

\begin{proof}[Proof of Theorem \ref{prop:MaxMetrics-Signed-Same-Peak-Set}(b)]
First, consider the case where $\{2, n-1\} \not \subseteq S$. If $2 \not \in S$, then since $\id^*[S](1) = n$ and $\bar\id^*[S](1) = -n$, we have $d_{\ell}(\id^*[S],\bar \id^*[S])=2n$. Similarly, if $n-1 \not \in S$ then $\textbf{e}[S](n) = n$ and $\bar \id[S](n) = -n$, hence $d_{\ell}(\textbf{e}[S],\bar\id[S])=2n$.

   Assume now $\{2, n-1\} \subseteq S$ and let $\sigma, \rho \in P^B(S;n)$ be distinct. We first claim that $d_\ell(\sigma, \rho) \leq 2n-1$. In fact, given any permutation $\tau\in P^B(S; n)$, $n$ must not appear at index $1$ or at index $n$ since the indices $2$ and $n-1$ are peaks. Notice that if there is $1\le j\le n$ such that $\tau(j)=n$, $j$ must be a peak. Moreover, if instead there is $1\le k\le n$ such that $\sigma(k)=-n$, $k$ cannot be a peak. Hence,  $d_\ell(\sigma, \rho) \neq 2n$, as the only way to obtain $2n$ would be to have $\bar{n}$ and $n$ appear in the same index in $\sigma$ and $\rho$, respectively. Thus, $d_\ell(\sigma, \rho) \leq 2n-1$. To see that this bound is achieved, consider the permutations $\id^*[S]$ and $\bar \id^*[S]$. Since $\id^*[S](1) = n-1$ and $\bar \id^*[S](1) = -n$, then $d_\ell(\textbf{e}[S],\bar \id[S]) = 2n-1$.
\end{proof}

Now are left with proving Theorem \ref{prop:MaxMetrics-Signed-Same-Peak-Set}(c). This will be done by providing an example of a pair of permutations in $P^B(S;n)$ whose word distance is $n^2-|S|$ (Lemma \ref{L:attain}), and by showing that $n^2-|S|$ is an upper bound to the word metric between any two permutations in $P^B(S;n)$ (see Theorem \ref{T:maxword}).

\begin{lemma}\label{L:attain} Let $n\ge 2$ and $S\subset[n]$ with $2|S|+1\le n$. 
Let 
\[
\sigma(j)=\begin{cases} n-j+1, \text{ if } j\notin S \text{ and } \ j+1\notin S\\
n-j+2, \text{ if } j\in S\\
-(n-j), \text{ if } j+1\in S,
\end{cases}
\]

\[
\pi(j)=\begin{cases}
    -(n-j+1), \text{ if } j\notin S\\
    n-j+1, \text{ if } j\in S.    
\end{cases}
\]
Then $\pi, \sigma \in P^B(S;n)$ and $d_W(\sigma,\pi)=n^2-|S|.$
\end{lemma}
\begin{proof}
It is straightforward to observe that $\sigma$ and $\pi$ have peak set $S$. Notice that 
\[
\pi^{-1}(k)=\begin{cases}
    -(n-k+1), \text{ if } n-k+1\notin S\\
    n-k+1, \text{ if } n-k+1\in S.
\end{cases}
\]
Computing directly, we obtain that
\[
\pi^{-1}\sigma(j)=\begin{cases}
    -j, \text{ if } j\notin S, j+1\notin S\\
    -j+1, \text{ if } j\in S\\
    -j-1, \text{ if } j+1\in S.
\end{cases}
\]
Hence,
\begin{align*}d_W(\sigma,\pi)&=d_W(\pi^{-1}\sigma,\textbf{e})=\sum_{j\notin S, \ j+1\notin S}(2j-1)+\sum_{j+1\in S}(2j-2)+\sum_{j\in S}(2j-1)\\
&=2\sum_{i=1}^n i-(n-2|S|)-3|S|=n(n+1)-n-|S|\\
&=n^2-|S|.
\end{align*}
\end{proof}

\begin{example}
When $n=8$ and $S=\{2,5\}$ we have $\sigma=\bar{7}86\bar{4}5321$, $\pi=\bar{8}7\bar{6}\bar{5}4\bar{3}\bar{2}\bar{1}$, and $\pi^{-1}=\bar{8}\bar{7}\bar{6}5\bar{4}\bar{3}2\bar{1}$ so that $\pi^{-1} \sigma=\bar{2}\bar{1}\bar{3}\bar{5}\bar{4}\bar{6}\bar{7}\bar{8}$. Note that $\pi^{-1}\sigma$ is $\overline{\textbf{e}}$, the longest element of $S_n^B$, right-multiplied by $s_{j-1}^B$ for each $j\in S$, i.e. $\bar\id[S]$.
\end{example}

\begin{lemma}\label{lem:distanceofopposites}
    Let $\alpha,\beta\in S^B_n$, and call $\overline{\alpha}$ the permutation in $S^B_n$ given by $\overline{\alpha}(j)=-\alpha(j)$. We have
\[
d_W(\alpha,\overline{\alpha})=d_W(\alpha,\beta)+d_W(\overline{\alpha},\beta).
\]
\end{lemma}

\begin{proof} By the left-invariance of the word metric, this is equivalent to showing
\[
d_W(\textbf{e},\overline{\textbf{e}})=d_W(\textbf{e},\alpha^{-1}\beta)+d_W(\overline{\textbf{e}},\alpha^{-1}\beta).
\]
We prove the following more general statement: given $\gamma\in S^B_n$,
\[
d_W(\textbf{e},\overline{\textbf{e}})=d_W(\textbf{e},\gamma)+d_W(\overline{\textbf{e}},\gamma).
\]
This, in turn, is equivalent to proving
\[
n^2=d_W(\textbf{e},\gamma)+d_W(\textbf{e},\overline{\gamma}).
\]
Recall that $d_W(\textbf{e},\gamma)=\ell_B(\gamma)$ for any $\gamma\in S^B_n$, where $\ell_B$ is the length of the signed permutation. As mentioned in Remark \ref{wordremark},  $\ell_B(\gamma)=\inv_B(\gamma)$, where 
\[
\inv_B(\gamma)=\negative(\gamma(1),\ldots,\gamma(n))+\inv(\gamma(1),\ldots,\gamma(n))+\nsp(\gamma(1),\ldots,\gamma(n))=J+K+L.\] 
Then 
\begin{align*}
d_W(e,\overline{\gamma})&=\inv_B(\overline{\gamma})=\negative(-\gamma(1),\ldots,-\gamma(n))+\inv(-\gamma(1),\ldots,-\gamma(n))+\nsp(-\gamma(1),\ldots,-\gamma(n))\\
&=(n-J)+ \left(\binom{n}{2}-K\right)+\left(\binom{n}{2}-L\right).
\end{align*}

Consequently, $d_W(\textbf{e},\gamma)+d_W(\textbf{e},\overline{\gamma})=n+2\binom{n}{2}=n^2$.
\end{proof}

\begin{theorem}\label{T:maxword}
    Given $\sigma,\pi\in P^{B}(S;n)$, we have $d_W(\sigma,\pi)\le n^2-|S|$.
\end{theorem}

\begin{proof}
Assume for the sake of contradiction that $d_W(\sigma,\pi)>n^2-|S|$, so $d_W(\sigma,\pi)\ge n^2-|S|+1$.

By the left-invariance of the word metric, we have $d_W(\pi,\overline{\pi})=d_W(\textbf{e},\pi^{-1}\overline{\pi})=d_W(\textbf{e},\overline{\textbf{e}})=n^2.$
Using Lemma \ref{lem:distanceofopposites}, 
\[
n^2=d_W(\pi,\overline{\pi})=d_W(\pi,\sigma)+d_W(\overline{\pi},\sigma)\ge n^2-|S|+1+d_W(\overline{\pi},\sigma).
\]
Hence, $d_W(\overline{\pi},\sigma)\le |S|-1,$ leading to a contradiction: $\sigma$ has a local maximum at each element in $S$ (and there are $|S|$ such elements), while $\overline{\pi}$ has a local minimum at each of those elements. Therefore, there is no way to do less than $|S|$ transpositions/sign changes to go from one to the other.
\end{proof}

\section{Future Directions}

There are many directions for further research on this topic, including the study of these metrics for the Coxeter group of type D. Another interesting direction would be to fix a peak set $S$ and enumerate how many pairs of signed permutations achieve the minimum and maximum for each metric. Similarly, one might consider exploring whether all intermediate values are achieved for each of the metrics. We also only considered a limited collection of metrics on signed permutations. One might consider looking at other metrics, such as looking at the word metric with respect to set of signed prefix reversals.

%%%%%%%%%%%%%%%%%%%%%%%%%%%%%%%%%%%%%%%%%%%%%%%%%%%%%%%%%%%%
%Bibliography
%%%%%%%%%%%%%%%%%%%%%%%%%%%%%%%%%%%%%%%%%%%%%%%%%%%%%%%%%%%%
\bibliographystyle{plain}
\bibliography{bibliography}

\end{document}